\documentclass[10pt]{article}
\usepackage{amsmath}
\usepackage{amsfonts}
\usepackage{amssymb}
\usepackage{amsthm}
\usepackage{enumerate}
\usepackage[totalwidth=450pt, totalheight=650pt]{geometry}

\begin{document}

\newcommand{\C}{{\mathbb{C}}}
\newcommand{\R}{{\mathbb{R}}}
\newcommand{\Z}{{\mathbb{Z}}}
\newcommand{\N}{{\mathbb{N}}}
\newcommand{\Q}{B}
\newcommand{\q}{\left}
\newcommand{\w}{\right}
\newcommand{\Ninf}{\N_\infty}
\newcommand{\Vol}[1]{\mathrm{Vol}\q(#1\w)}
\newcommand{\B}[4]{B_{\q(#1,#2\w)}\q(#3,#4\w)}
\newcommand{\CjN}[3]{\q\|#1\w\|_{C^{#2}\q(#3\w)}}
\newcommand{\Cj}[2]{C^{#1}\q( #2\w)}
\newcommand{\grad}{\bigtriangledown}
\newcommand{\sI}[2]{\mathcal{I}\q(#1,#2 \w)}
\newcommand{\Det}[1]{\det_{#1\times #1}}
\newcommand{\sK}{\mathcal{K}}
\newcommand{\sKt}{\widetilde{\mathcal{K}}}
\newcommand{\sA}{\mathcal{A}}
\newcommand{\sB}{\mathcal{B}}
\newcommand{\sC}{\mathcal{C}}
\newcommand{\sD}{\mathcal{D}}
\newcommand{\sS}{\mathcal{S}}
\newcommand{\sF}{\mathcal{F}}
\newcommand{\sT}{\mathcal{T}}
\newcommand{\sQ}{\mathcal{Q}}
\newcommand{\sV}{\mathcal{V}}
\newcommand{\cV}{\q( \sV\w)}
\newcommand{\vsig}{\varsigma}
\newcommand{\vsigt}{\widetilde{\vsig}}
\newcommand{\dil}[2]{#1^{\q(#2\w)}}
\newcommand{\lA}{-\log_2 \sA}
\newcommand{\eh}{\widehat{e}}
\newcommand{\Ho}{\mathbb{H}^1}
\newcommand{\sd}{\sum d}
\newcommand{\dt}{\tilde{d}}
\newcommand{\dhc}{\hat{d}}
\newcommand{\Span}[1]{\mathrm{span}\q\{ #1 \w\}}
\newcommand{\dspan}[1]{\dim \Span{#1}}
\newcommand{\K}{K_0}
\newcommand{\ad}[1]{\mathrm{ad}\q( #1 \w)}
\newcommand{\LtOpN}[1]{\q\|#1\w\|_{L^2\rightarrow L^2}}
\newcommand{\LpOpN}[2]{\q\|#2\w\|_{L^{#1}\rightarrow L^{#1}}}
\newcommand{\LplqOpN}[3]{\q\|#3\w\|_{L^{#1}\q(\ell^{#2}\q(\N^\nu\w)\w)\rightarrow L^{#1}\q(\ell^{#2}\q(\N^\nu\w)\w)}}
\newcommand{\LpN}[2]{\q\|#2\w\|_{L^{#1}}}
\newcommand{\LplqN}[3]{\q\|#3\w\|_{L^{#1}\q(\ell^{#2}\q(\N^\nu\w)\w)}}
\newcommand{\Jac}{\mathrm{Jac}\:}
\newcommand{\kapt}{\widetilde{\kappa}}
\newcommand{\gt}{\widetilde{\gamma}}
\newcommand{\gtt}{\widetilde{\widetilde{\gamma}}}
\newcommand{\gh}{\widehat{\gamma}}
\newcommand{\Sh}{\widehat{S}}
\newcommand{\Wh}{\widehat{W}}
\newcommand{\Ih}{\widehat{I}}
\newcommand{\Wt}{\widetilde{W}}
\newcommand{\Xt}{\widetilde{X}}
\newcommand{\Tt}{\widetilde{T}}
\newcommand{\Nt}{\widetilde{N}}
\newcommand{\Phit}{\widetilde{\Phi}}
\newcommand{\Vh}{\widehat{V}}
\newcommand{\Xh}{\widehat{X}}
\newcommand{\sSh}{\widehat{\mathcal{S}}}
\newcommand{\sFh}{\widehat{\mathcal{F}}}
\newcommand{\thetah}{\widehat{\theta}}
\newcommand{\ct}{\widetilde{c}}
\newcommand{\at}{\tilde{a}}
\newcommand{\bt}{\tilde{b}}
\newcommand{\fg}{\mathfrak{g}}
\newcommand{\cC}{\q( \sC\w)}
\newcommand{\cG}{\q( \sC_{\fg}\w)}
\newcommand{\cJ}{\q( \sC_{J}\w)}
\newcommand{\cY}{\q( \sC_{Y}\w)}
\newcommand{\cZ}{\q( \sC_{Z}\w)}
\newcommand{\cYu}{\q( \sC_{Y}^u\w)}
\newcommand{\cGu}{\q( \sC_{\fg}^u\w)}
\newcommand{\cJu}{\q( \sC_{J}^u\w)}
\newcommand{\cZu}{\q( \sC_{Z}^u\w)}
\newcommand{\cYs}{\q( \sC_{Y}^s\w)}
\newcommand{\cGs}{\q( \sC_{\fg}^s\w)}
\newcommand{\cJs}{\q( \sC_{J}^s\w)}
\newcommand{\cZs}{\q( \sC_{Z}^s\w)}
\newcommand{\Bb}{\overline{B}}
\newcommand{\Qb}{\overline{\Q}}
\newcommand{\sP}{\mathcal{P}}
\newcommand{\sY}{\mathcal{Y}}
\newcommand{\sN}{\mathcal{N}}
\newcommand{\cH}{\q(\mathcal{H}\w)}
\newcommand{\Omegat}{\widetilde{\Omega}}
\newcommand{\Lie}{\mathrm{Lie}}
\newcommand{\thetat}{\widetilde{\theta}}
\newcommand{\etat}{\widetilde{\eta}}
\newcommand{\psit}{\widetilde{\psi}}
\newcommand{\supp}[1]{\mathrm{supp}\q(#1\w)}
\newcommand{\sM}{\mathcal{M}}
\newcommand{\sMt}{\widetilde{\mathcal{M}}}

\newcommand{\Lpp}[2]{L^{#1}\q(#2\w)}
\newcommand{\Lppn}[3]{\q\|#3\w\|_{\Lpp{#1}{#2}}}

\newtheorem{thm}{Theorem}[section]
\newtheorem{cor}[thm]{Corollary}
\newtheorem{prop}[thm]{Proposition}
\newtheorem{lemma}[thm]{Lemma}
\newtheorem{conj}[thm]{Conjecture}

\theoremstyle{remark}
\newtheorem{rmk}[thm]{Remark}

\theoremstyle{definition}
\newtheorem{defn}[thm]{Definition}

\theoremstyle{remark}
\newtheorem{example}[thm]{Example}

\theoremstyle{remark}
\newtheorem{step}{Step}

\numberwithin{equation}{section}

\title{Multi-parameter singular Radon transforms}
\author{Elias M. Stein and Brian Street\footnote{The second author was partially supported by NSF DMS-0802587.}}
\date{}

\maketitle

\begin{abstract}
The purpose of this announcement is to describe a development
given in a series of forthcoming papers by
the authors that concern operators of the form
\begin{equation*}
f\mapsto \psi\q(x\w) \int f\q(\gamma_t\q(x\w)\w) K\q(t\w)\: dt,
\end{equation*}
where $\gamma_t\q(x\w)=\gamma\q(t,x\w)$ is a $C^\infty$ function defined
on a neighborhood of the origin in $\q(t,x\w)\in \R^N\times \R^n$
satisfying $\gamma_0\q(x\w)\equiv x$, $K\q(t\w)$ is a ``multi-parameter
singular kernel'' supported near $t=0$, and $\psi$ is a cutoff function
supported near $x=0$.  This note concerns the case when $K$
is a ``product kernel.''  The goal is to give conditions
on $\gamma$ such that the above operator is bounded on $L^p$
for $1<p<\infty$.  
Associated maximal functions are also discussed.
The ``single-parameter'' case
when $K$ is a Calder\'on-Zygmund kernel was studied by
Christ, Nagel, Stein, and Wainger.
The theory here extends these results to the multi-parameter context
and also deals effectively with the case when $\gamma$ is
real-analytic.

\end{abstract}

\section{Introduction}
The purpose of this note is to announce the results from a three
part series of papers by the authors 
\cite{StreetMultiParameterSingRadonLt,
SteinStreetMultiParameterSingRadonLp,
StreetMultiParameterSingRadonAnal}, and to give an overview
of the main ideas in a somewhat simpler context.
The object is to study the $L^p$ ($1<p<\infty$) boundedness of operators of the form
\begin{equation}\label{EqnIntroT}
Tf\q(x\w) = \psi\q(x\w) \int f\q(\gamma_t\q(x\w)\w) K\q(t\w)\: dt,
\end{equation}
where $\gamma_t\q(x\w)=\gamma\q(t,x\w):\R^N_0\times \R^n_0\rightarrow\R^n$
is a $C^\infty$ function with $\gamma_0\q(x\w)\equiv x$, $\psi\in C_0^\infty\q(\R^n\w)$ is supported
on a small neighborhood of $0\in \R^n$, and $K$ is an appropriate ``multi-parameter''
singular kernel supported on a small neighborhood of $0\in \R^N$.
Here we have written $\R^N_0$ for a neighborhood of the origin
in $\R^N$.
The case when $K$ is a Calder\'on-Zygmund kernel (that is,
when the number of parameters, $\nu$, equals $1$) was studied
in the paper of Christ, Nagel, Stein, and Wainger
\cite{ChristNagelSteinWaingerSingularAndMaximalRadonTransforms}.
Given the complexity of the formulation and proof of our general
results, we will here restrict ourselves to the case
when $K$ is a ``product kernel'' and give only an outline
of the main points of the argument contained
in
\cite{StreetMultiParameterSingRadonLt,
SteinStreetMultiParameterSingRadonLp,
StreetMultiParameterSingRadonAnal}.
Some of these results, valid in a more general setting,
are indicated in Section \ref{SectionGeneral}.
%

In the definition of a product kernel given in Section
\ref{SectionProductKernel},
we fix a decomposition of $\R^N$ into $\nu$ factors,
$\R^N=\R^{N_1}\times \cdots\times \R^{N_\nu}$, and
write $t=\q(t_1,\ldots, t_\nu\w)$. 
The main goal of the papers 
\cite{StreetMultiParameterSingRadonLt,SteinStreetMultiParameterSingRadonLp}
 is to give conditions on $\gamma$ for which 
the operator $T$ given by \eqref{EqnIntroT} is bounded on $L^p$ ($1<p<\infty$)
for {\it every} such product kernel (supported on a sufficiently small
neighborhood of $0\in \R^N$).
Under these same conditions, one also obtains the $L^p$ boundedness
of the corresponding maximal function:
\begin{equation}\label{EqnIntroM}
\sM f\q(x\w) = \psi\q(x\w) \sup_{0<\delta_1,\ldots, \delta_\nu<<1} \int_{\q|t\w|\leq 1} \q|f\q(\gamma_{\delta_1 t_1,\ldots, \delta_\nu t_\nu}\q(x\w)\w)\w|\: dt.
\end{equation}
To simplify the presentation we limit ourselves to the special case when
$\gamma$ is given by
\begin{equation*}
\gamma_t\q(x\w)=\exp\q(\sum_{0<\q|\alpha\w|\leq L} t^\alpha X_\alpha\w)x,
\end{equation*}
where the $X_\alpha$ are $C^\infty$ vector fields.  I.e., $\gamma$ is an
exponential of a {\it finite} sum of vector fields.
In our cited work,
more general $\gamma$ are considered.  See Section \ref{SectionGeneral} for some comments on this.

In \cite{StreetMultiParameterSingRadonAnal}, it is shown that when
$\gamma$ is assumed to be real analytic, many of the assumptions
necessary to carry out to proofs 
in our work
hold automatically.  This is discussed in Section \ref{SectionRealAnal}.

\subsection{Statement of main results}\label{SectionResults}
In this section, we state the main results from this paper.
Consider a $C^\infty$ function
\begin{equation}\label{EqnGammaFiniteExp}
\gamma_t\q(x\w) = \exp\q(\sum_{0<\q|\alpha\w|\leq L}t^\alpha X_\alpha\w) x,
\end{equation}
where the $X_\alpha$ are $C^\infty$ vector fields on $\R^n$.  Here, and in what follows,
we will be restricting attention to $\q(t,x\w)$
in some small neighborhood of the origin $\R^N\times \R^n$.

Decompose $\R^N=\R^{N_1}\times \cdots \times \R^{N_\nu}$ as in the introduction.
We will be considering product kernels $K\q(t_1,\ldots, t_\nu\w)$ relative
to this decomposition of $\R^N$ (see Definition \ref{DefnProdKer}).
For a multi-index $\alpha\in \N^{N}$, we obtain a corresponding decomposition
$\alpha=\q(\alpha_1,\ldots, \alpha_\nu\w)$ where $\alpha_\mu\in \N^{N_\mu}$
and $t^{\alpha} = \prod_\mu t_\mu^{\alpha_\mu}$.

Our goal is to give conditions on the vector fields $X_\alpha$
such that the operator $T$ given by \eqref{EqnIntroT}
and the operator $\sM$ given by \eqref{EqnIntroM}
are bounded on $L^p$ ($1<p<\infty$).

Corresponding to a multi-index $\alpha\in \N^\nu$, we assign
a formal ``degree'' $\deg\q(\alpha\w)=\q(\q|\alpha_1\w|,\ldots, \q|\alpha_\nu\w|\w)\in \N^\nu$.  This assigns to each vector field $X_\alpha$ the formal
degree $\deg\q(\alpha\w)$.  We say $\alpha$ is a ``pure power''
if $\deg\q(\alpha\w)$ is non-zero in precisely one component; otherwise
we say $\alpha$ is a ``non-pure power.''
This yields two finite sets of vector fields, paired with formal degrees:
\begin{equation*}
\begin{split}
\sP &:= \q\{\q(X_\alpha, \deg\q(\alpha\w)\w): \alpha\text{ is a pure power}\w\},\\
\sN &:= \q\{\q(X_\alpha, \deg\q(\alpha\w)\w): \alpha\text{ is a non-pure power}\w\}.\\
\end{split}
\end{equation*}
Let $\sS$ be the smallest set of vector fields paired with formal degrees such
that
\begin{itemize}
\item $\sP\subseteq \sS$ and
\item if $\q(X,d\w),\q(Y,e\w)\in \sS$, then $\q(\q[X,Y\w],d+e\w)\in \sS$.
\end{itemize}

We are now prepared to state the two assumptions we make on the vector
fields $X_\alpha$.
\begin{itemize}
\item A ``finite-type'' condition:  there exists a finite subset $\sF\subseteq \sS$ such that for every $\q(X,d\w)\in \sS$, we have
\begin{equation}\label{EqnFiniteTypeControl}
X = \sum_{\substack{\q(Y,e\w)\in \sF \\ e\leq d}} c_X^Y Y,
\end{equation}
where $c_X^Y\in C^\infty$, and $e\leq d$ means that the inequality
holds coordinatewise (that is, $e_\mu\leq d_\mu$ for each $\mu$).

\item An ``algebraic'' condition:  for every $\q(X,d\w)\in \sN$, there
exists a finite subset $\sF\subseteq \sS$ such that
\begin{equation}\label{EqnAlgControl}
X = \sum_{\substack{\q(Y,e\w)\in \sF \\ e\leq d}} c_X^Y Y,
\end{equation}
where $c_X^Y\in C^\infty$.  Note that, if the finite type condition holds,
one may take $\sF$ as in that condition.
\end{itemize}

\begin{thm}\label{ThmMainThm}
Under the above two conditions, the operator $T$ given by \eqref{EqnIntroT} (for any product kernel $K\q(t_1,\ldots, t_\nu\w)$ with small support)
and the operator $\sM$ given by \eqref{EqnIntroM} are bounded
on $L^p$ ($1<p<\infty$).
\end{thm}

\begin{rmk}
A more general version of our result can be found in our cited
work, where the following issues are addressed:
\begin{itemize}
\item $\gamma$ is not restricted to be of the form \eqref{EqnGammaFiniteExp} (see Section \ref{SectionGeneral}).

\item A more general class of kernels than product kernels is studied.

\item The sums \eqref{EqnFiniteTypeControl} and \eqref{EqnAlgControl}
are replaced with more general (and more complicated) conditions.
\end{itemize}
\end{rmk}

\begin{rmk}
Our general results for the case $\nu=1$ in fact have wider
scope than those in
\cite{ChristNagelSteinWaingerSingularAndMaximalRadonTransforms}.
See also the comments in Section \ref{SectionGeneral}.
Note that when $\nu=1$ the algebraic condition \eqref{EqnAlgControl}
is vacuous, and the finite type condition follows
from H\"ormander's condition:  that the Lie algebra generated
by the $\q\{X_\alpha\w\}$ spans the tangent space at $0$.
We remark that when $\nu>1$, H\"ormander's condition implies
neither the finite type condition, nor the algebraic condition.
For example, decompose $\R^2=\R\times \R$ and write $\q(t_1,t_2\w)\in \R\times \R$.  Then for $x\in \R$ ($n=1$)
\begin{equation}\label{EqnExNoAlg}
\gamma_{t_1,t_2}\q(x\w)=\exp\q(t_1^3 \frac{\partial}{\partial x} + t_1 t_2 \frac{\partial}{\partial x} + t_2^3 \frac{\partial}{\partial x}\w)x=x+t_1^3+t_1t_2+t_2^3
\end{equation}
does not satisfy the algebraic condition.
This would have been
true, as well, if we had taken $t_1^2$ and $t_2^2$ instead of $t_1^3$ and $t_2^3$.  However, the more general algebraic condition in our cited work
allows for the case of $t_1^2$ and $t_2^2$.  Furthermore, there
exist product kernels $K\q(t_1,t_2\w)$ such that if
$\gamma$ is given by \eqref{EqnExNoAlg}, then operator $T$ given
by \eqref{EqnIntroT} is not bounded on $L^2$ (which is not true
if we use $t_1^2$ and $t_2^2$)--see \cite{StreetMultiParameterSingRadonLt}.
For $\q(x,y\w)\in\R^2$, if we set $X_1=\frac{\partial}{\partial x}$, $X_2=e^{-\frac{1}{x^2}}\frac{\partial}{\partial y}$, and $X_3=\frac{\partial}{\partial y}$,
then
\begin{equation*}
\exp\q(t_1 X_1 + t_1^2 X_2 + t_2 X_3\w)\q(x,y\w)
\end{equation*}
does not satisfy the finite type condition, while $X_1$, $X_2$, and $X_3$
clearly satisfy H\"ormander's condition.
\end{rmk}


\begin{rmk}
If the finite type condition holds, then the involutive distribution\footnote{Here,
we are using distribution to mean a $C^\infty$ module of vector fields, and an involutive distribution is one which is closed under Lie brackets.}
generated by the vector fields in $\sP$ is finitely generated as a $C^\infty$
module (the vector fields in $\sF$ generate it).
In this context, the Frobenius theorem applies to foliate the manifold
into leaves.
Further, if the algebraic condition also holds, then the vector fields
corresponding to the non-pure powers are tangent to these leaves. 
More is true:  the finite-type and algebraic conditions can be thought
of as ``scale invariant'' versions of the above.  This is discussed
in Section \ref{SectionScaling}.  In the single parameter case ($\nu=1$)
this scale invariance comes ``for free.''  However, in the multi-parameter
case ($\nu>1$) this scale invariance is an essential point.
\end{rmk}

\begin{rmk}
The algebraic condition excludes examples like
\begin{equation*}
f\mapsto \psi\q(x\w) \int f\q(x-st\w)K\q(s,t\w)\: ds\:dt,
\end{equation*}
where $K\q(s,t\w)$ is a product kernel corresponding
to $\q(s,t\w)\in \R^2=\R\times \R$ (here, $x\in \R$).  It is well known
that there exist product kernels such that the above
operator is not bounded on $L^2$.  This dates
back to \cite{NagelWaingerL2BoundednessOfHilbertTransformsMultiParameterGroup}.
See \cite{StreetMultiParameterSingRadonLt} for a further
discussion.
\end{rmk}

\begin{rmk}
Some other work that overlaps with our results is in the recent work
``Singular integrals with flag kernels on graded groups: I''
by Nagel, Ricci, Stein, and Wainger.
Also, earlier work by Carbery, Wright, and Wainger
\cite{CarberyWaingerWrightDoubleHilbertTransformsAlongPolySurfInRRR,
CarberyWaingerWrightSingularIntegralsAndTheNewtonDiagram,
CarberyWaingerWrightTripleHilbertTransformsAlongPolySurfInRRRR}.
Some other related works include 
\cite{ChoHongKimYangTripleHilbertTransformsAlongPolynomialSurfaces,RicciSteinMultiparameterSingularIntegralsAndMaximalFunctions}.
\end{rmk}

\section{When $\gamma$ is real analytic}\label{SectionRealAnal}
The third paper in the series \cite{StreetMultiParameterSingRadonAnal}
deals with the special case when $\gamma$ is assumed to be real
analytic (in {\it both} variables). 
Once again, we assume in this note that $\gamma$
is given by \eqref{EqnGammaFiniteExp}.
In this case
assuming that $\gamma$ is real analytic is the same
as assuming that each $X_\alpha$ is real analytic.
The essential point in this situation is the following proposition.
\begin{prop}\label{PropGammaRealAnalFiniteType}
When $\gamma$ is real analytic (equivalently when each $X_\alpha$
is real analytic), the finite type condition holds automatically.
\end{prop}

Proposition \ref{PropGammaRealAnalFiniteType} is a consequence
of the Weierstrass preparation theorem.\footnote{This is closely
related to the fact that the involutive distribution generated
by a finite collection of real analytic vector fields is automatically
locally finitely generated as a $C^\infty$ module and the Frobenius theorem
applies.  This dates back to \cite{NaganoLinearDifferentialSystemsWithSingularities, LobryControlabiliteDesSystemesNonLinearies}.}
We refer
the reader to \cite{StreetMultiParameterSingRadonAnal}
for details.

\begin{cor}\label{CorTBoundedRealAnal}
If each $X_\alpha$ is real analytic, and the algebraic condition holds,
then $T$ and $\sM$ are bounded on $L^p$ ($1<p<\infty$).
\end{cor}
\begin{proof} 
This follows directly from combining Theorem \ref{ThmMainThm} and Proposition \ref{PropGammaRealAnalFiniteType}.
\end{proof}

\begin{cor}\label{CorRealAnalnu1}
If $\nu=1$ (i.e., when $K$ is a Calder\'on-Zygmund kernel) and if each $X_\alpha$ is real analytic, then
$T$ and $\sM$ are bounded on $L^p$ ($1<p<\infty$).
\end{cor}
\begin{proof}
This follows from Corollary \ref{CorTBoundedRealAnal} since the
algebraic condition is vacuous when $\nu=1$.
\end{proof}

Actually, more is true for the maximal function.  Indeed, we have the following
result.
\begin{thm}\label{ThmRealAnalMax}
For any $\nu>0$, if each $X_\alpha$ is real analytic, then $\sM$ is bounded on $L^p$ ($1<p\leq \infty$).
\end{thm}
\begin{proof}[Proof idea]
Consider,
\begin{equation*}
\sM f\q(x\w)  \lesssim \psi\q(x\w) \sup_{j_1,\ldots, j_\nu\in \N}
\int_{\q|t\w|\leq a}
 \q|f\q(\exp\q\{\sum_{0<\q|\alpha\w|\leq L} 2^{-j_1\q|\alpha_1\w|}\cdots 2^{-j_\nu\q|\alpha_\nu\w|} t^{\alpha} X_\alpha\w\}x\w)\w| \: dt,
\end{equation*}
where $a>0$ is some small number (depending on the $X_\alpha$).
Following an idea of Christ 
\cite{ChristTheStrongMaximalFunctionOnANilpotentGroup}, 
we study the stronger maximal operator
\begin{equation*}
\sMt f\q(x\w) = \psi\q(x\w) 
\sup_{k_\alpha\in \N}
\int_{\q|t\w|\leq a}
 \q|f\q(\exp\q\{\sum_{0<\q|\alpha\w|\leq L} 2^{-k_\alpha} t^{\alpha} X_\alpha\w\}x\w)\w| \: dt.
\end{equation*}
It is easy to see that $\sM f\q(x\w)\lesssim \sMt f\q(x\w)$, since
one may take $k_\alpha=j_1\q|\alpha_1\w|+\cdots+j_\nu\q|\alpha_\nu\w|$.  Furthermore,
because of the form of $\sMt$, each $X_\alpha$ behaves like
a vector field corresponding to a pure power.  
Thus, the algebraic condition holds automatically.
Since the finite type condition holds by Proposition \ref{PropGammaRealAnalFiniteType},
the proof of Theorem \ref{ThmMainThm} (with minor modifications) goes through to prove
that $\sMt$ is bounded on $L^p$.
\end{proof}

The above results generalize to the case when $\gamma$ is merely
assumed to be a germ of a real analytic function
satisfying $\gamma_0\q(x\w)\equiv x$.  I.e., $\gamma$
not necessarily of the form
\eqref{EqnGammaFiniteExp}.
We take a moment, here, to discuss the the generalizations
of Corollary \ref{CorRealAnalnu1} and Theorem \ref{ThmRealAnalMax}.
Proposition \ref{PropGammaRealAnalFiniteType} and Corollary \ref{CorTBoundedRealAnal} can also be generalized, but we defer a discussion of this
to \cite{StreetMultiParameterSingRadonAnal}.

\begin{thm}\label{ThmGenRealAnalT}
Suppose $\nu=1$ (i.e., $K$ is a Calder\'on-Zygmund kernel)
and $\gamma$ is a real analytic function with $\gamma_0\q(x\w)\equiv x$.
Then the operator $T$ given by \eqref{EqnIntroT} is bounded on 
$L^p$ ($1<p<\infty$).
\end{thm}

\begin{thm}\label{ThmGenRealAnalM}
Suppose $\gamma$ is a real analytic function
satisfying $\gamma_0\q(x\w)\equiv x$.  Then the operator $\sM$
given by \eqref{EqnIntroM} (for any $\nu$) is bounded on $L^p$ ($1<p\leq \infty$).
\end{thm}

Theorem \ref{ThmGenRealAnalM} generalizes some results from the literature.
A well known result of Bourgain 
\cite{BourgainARemarkOnTheMaximalFunctionAssociatedToAnAnalyticVectorField}
deals with the special case of Theorem \ref{ThmGenRealAnalM}
when $\nu=1$, $p=2$, and $\gamma_t\q(x\w) = x+tv\q(x\w)$, where
$x\in \R^2$, $t\in \R$, and $v$ is a real analytic vector field on $\R^2$--Theorem
\ref{ThmGenRealAnalM} extends this to $1<p\leq \infty$ and
allows $x\in \R^n$ for any $n$ (instead of only $n=2$).
Also, a result of Christ \cite{ChristTheStrongMaximalFunctionOnANilpotentGroup}
deals with a special case of 
Theorem \ref{ThmRealAnalMax}
where each $X_\alpha$ is
a left invariant vector field on a nilpotent Lie group (and therefore real analytic).
In fact, the proof of Theorem \ref{ThmGenRealAnalM} incorporates
ideas both from 
\cite{ChristTheStrongMaximalFunctionOnANilpotentGroup}
and 
\cite{BourgainARemarkOnTheMaximalFunctionAssociatedToAnAnalyticVectorField}.
We refer the reader to these references 
and \cite{StreetMultiParameterSingRadonAnal}
for further details on these issues.

\section{Product Kernels}\label{SectionProductKernel}
In this section, we discuss the notion of a product kernel.
Our main reference for the following is 
\cite{NagelRicciSteinSingularIntegralsWithFlagKernels}
and we refer the reader there for further information.

We fix a decomposition $\R^N=\R^{N_1}\times \cdots \times \R^{N_\nu}$
as in the introduction.  Our goal is to define the notion of a product
kernel $K\q(t_1,\ldots, t_\nu\w)$, relative to this decomposition.

\begin{defn}
A $k$-normalized bump function on $\R^N$ is a $C^k$ function
supported on the unit ball with $C^k$ norm bounded by $1$.
\end{defn}

The definitions that follow turn out to be independent of the
choice of $k\geq 1$, and we therefore refer to normalized bump functions (which
can be taken to mean $1$-normalized bump functions, for instance).

\begin{defn}\label{DefnProdKer}
A product kernel relative
to the decomposition $\R^N=\R^{N_1}\times \cdots \times \R^{N_\nu}$
is a distribution $K\q(t_1,\ldots,t_\nu\w)$, which corresponds
with a $C^\infty$ function away from the coordinate axes $t_\mu=0$
and which satisfies:
\begin{itemize}
\item For each multi-index $\alpha=\q(\alpha_1,\ldots, \alpha_\nu\w)$,
there is a constant $C_\alpha$ such that
\begin{equation*}
\q|\partial_{t_1}^{\alpha_1} \cdots \partial_{t_\nu}^{\alpha_\nu} K\q(t\w)\w|\leq C_\alpha \q|t_1\w|^{-N_1-\q|\alpha_1\w|} \cdots \q|t_\nu\w|^{-N_{\nu}-\q|\alpha_\nu\w|}.
\end{equation*}

\item We proceed recursively on $\nu$:
\begin{itemize}
\item For $\nu=1$, given any normalized bump function $\phi$ and and $R>0$,
\begin{equation*}
\int K\q(t\w) \phi\q(Rt\w) \: dt
\end{equation*}
is bounded independently of $\phi$ and $R$.

\item For $\nu>1$, given any $1\leq \mu\leq \nu$, any normalized bump
function $\phi$ on $\R^{N_\mu}$, and any $R>0$, the distribution
\begin{equation*}
K_{\phi,R}\q(t_1,\ldots, t_{\mu-1}, t_{\mu+1},\ldots, t_\nu\w) = \int K\q(t\w) \phi\q(R t_\mu\w) \: dt_\mu
\end{equation*}
is a product kernel on the lower dimensional space $\R^{N_1}\times \cdots\times \R^{N_{\mu-1}}\times \R^{N_{\mu+1}}\times \cdots \times \R^{N_\nu}$, uniformly
in $R$ and $\phi$.
\end{itemize}
\end{itemize}
\end{defn}

\begin{rmk}
The setting in \cite{NagelRicciSteinSingularIntegralsWithFlagKernels} was slightly
more general:  product kernels were defined relative to a decomposition
of $\R^N$ into subspaces which were homogeneous under certain dilations.
Here, we have chosen the standard dilations on each $\R^{N_\mu}$.
The framework in 
our cited work
takes into account these more general product kernels, and even kernels
which are not product kernels.  We do not discuss this here, though.
\end{rmk}

Given $\delta=\q(\delta_1,\ldots, \delta_\nu\w)\in \q[0,\infty\w)^\nu$,
define $\delta \q(t_1,\ldots, t_\nu\w) = \q(\delta_1 t_1,\ldots, \delta_\nu t_\nu\w)$.  Further, for $j=\q(j_1,\ldots, j_\nu\w)\in \Z^\nu$, define
$2^j=\q(2^{j_1},\ldots, 2^{j_\nu}\w)\in \q(0,\infty\w)^{\nu}$.
Finally, for a function $f\q(t\w)$ define
\begin{equation*}
\dil{f}{2^j}\q(t\w) = 2^{j_1N_1+\cdots + j_\nu N_\nu} f\q(2^j t\w).
\end{equation*}
Note that $\int \dil{f}{2^j}\q(t\w)\: dt = \int f\q(t\w)\: dt$.

Let $B^N\q(a\w)$ denote the ball in $\R^N$ of radius $a>0$.
We use the following characterization of product kernels,
which is a slight modification of a result of
\cite{NagelRicciSteinSingularIntegralsWithFlagKernels}.
\begin{prop}\label{PropDecompProdKer}
Suppose $\q\{\eta_j\w\}_{j\in \N^\nu}$ is a bounded subset of
$C_0^{\infty}\q(B^N\q(a\w)\w)$ satisfying
\begin{equation}\label{EqnProdKernelCancel}
\int \eta_j\q(t\w) \: dt_\mu = 0,\text{ if }j_\mu\ne 0.
\end{equation}
Then the sum
\begin{equation*}
\sum_{j\in \N^\nu} \dil{\eta_j}{2^j}
\end{equation*}
converges in distribution to a product kernel which is supported in $B^N\q(a\w)$.
Conversely, every product kernel supported in $B^N\q(a\w)$ can be decomposed in this way.
\end{prop}

Recall our object of study.  We are interested in operators of the form
\begin{equation*}
Tf\q(x\w) = \psi\q(x\w) \int f\q(\gamma_t\q(x\w)\w) K\q(t\w)\: dt,
\end{equation*}
where $K\q(t\w)$ is a product kernel supported in $B^N\q(a\w)$ for some
small $a>0$.
Proposition \ref{PropDecompProdKer} allows us to decompose
\begin{equation*}
K\q(t\w) = \sum_{j\in \N^\nu} \dil{\eta_j}{2^j}\q(t\w),
\end{equation*}
where $\q\{\eta_j\w\}_{j\in \N^\nu}\subset C_0^\infty\q(B^N\q(a\w)\w)$
is a bounded set, satisfying \eqref{EqnProdKernelCancel}.
This yields a corresponding decomposition of $T$.  Indeed,
define
\begin{equation}\label{EqnDefnTj}
T_j f\q(x\w) = \psi\q(x\w) \int f\q(\gamma_t\q(x\w)\w) \dil{\eta_j}{2^j}\q(t\w)\: dt.
\end{equation}
Then, $T=\sum_{j\in \N^\nu} T_j$.

\section{Scaling and the Frobenius theorem}\label{SectionScaling}
The heart of the proof that $T$ is bounded on $L^p$ ($1<p<\infty$)
is the idea that for $j,k\in \N^\nu$, $j\ne k$, $T_j$ and $T_k$
are essentially of the same form, but at different ``scales.''
Because they reside at different scales, we will obtain, for instance,
the almost orthogonality estimate,
\begin{equation}\label{EqnToShowAO}
\LtOpN{T_j^{*} T_k}, \LtOpN{T_j T_k^{*}} \lesssim 2^{-\epsilon\q|j-k\w|},
\end{equation}
for some $\epsilon>0$.  The $L^2$ boundedness of $T$ then follows
immediately from the Cotlar-Stein lemma.

The key to proving this almost orthogonality is to use an appropriate
scaling map which is adapted to the vector fields.
This scaling map was developed in \cite{StreetMultiParameterCCBalls},
based on ideas from
\cite{NagelSteinWaingerBallsAndMetricsDefinedByVectorFields,
TaoWrightLpImprovingBoundsForAverages}.
As was pointed out in \cite{StreetMultiParameterCCBalls}, and as will
be discussed below, these scaling maps can be viewed as the coordinate
charts defining the leaves in a quantitative version of the
Frobenius theorem on involutive distributions.

Let $\sF$ be the finite set of vector fields with formal degrees given in the finite type condition.
Enumerate the set $\sF$ to obtain a finite list of vector fields with
formal degrees $\q(X_1,d_1\w),\ldots, \q(X_q,d_q\w)$.
Notice, by the finite type assumption,
we have
\begin{equation}\label{EqnNSWInvol}
\q[X_j,X_k\w]=\sum_{d_l\leq d_j+d_k} c_{j,k}^l X_l,
\end{equation}
where $d_l\leq d_j+d_k$ denotes that the inequality holds coordinatewise,
and $c_{j,k}^l\in C^\infty$.

Let $\sD$ be the distribution\footnote{Here, we are using distribution to mean
$C^\infty$ module of vector fields.} generated by $X_1,\ldots, X_q$.  Note that $\sD$
is involutive:  if $X,Y\in \sD$, then $\q[X,Y\w]\in \sD$.
This follows directly from \eqref{EqnNSWInvol}.
Because of this, the classical Frobenius theorem applies to foliate
the ambient space into leaves; the tangent bundle to each
leaf given by $\sD$.

However, more is true.  Let $\delta\in\q[0,1\w]^\nu$ and multiply both
sides of \eqref{EqnNSWInvol} by $\delta^{d_j+d_k}$ to obtain,
\begin{equation*}
\q[\delta^{d_j} X_j, \delta^{d_k}X_k\w] = \sum_{d_l\leq d_j+d_k} \q(\delta^{d_j+d_k-d_l} c_{j,k}^l\w) \delta^{d_l} X_l.
\end{equation*}
Defining
\begin{equation*}
c_{j,k}^{l,\delta}=\begin{cases}
\delta^{d_j+d_k-d_l}c_{j,k}^l & \text{if }d_l\leq d_j+d_k\\
0 & \text{otherwise,}
\end{cases}
\end{equation*}
we have that $c_{j,k}^{l,\delta}\in C^\infty$ uniformly in $\delta$, and
\begin{equation*}
\q[\delta^{d_j} X_j, \delta^{d_k} X_k\w] = \sum_{l} c_{j,k}^{l,\delta} \delta^{d_l} X_l.
\end{equation*}
Because of this one might hope that the Frobenius theorem holds when applied
to the vector fields $\delta^{d_j}X_j$ ``uniformly'' in $\delta$ in the sense
that appropriate coordinate charts defining the leaves may be controlled
``uniformly'' in $\delta$.  Indeed, this turns out to be the case
and is the main theorem of \cite{StreetMultiParameterCCBalls}.

We turn to describing this theorem now.  Let $Z_j=\delta^{d_j} X_j$
so that
\begin{equation*}
\q[Z_j,Z_k\w]=\sum_{l} \ct_{j,k}^l Z_l,
\end{equation*}
with $\ct_{j,k}^l\in C^\infty$ uniformly in $\delta$.
Fix $x_0\in \R^n$.  Our goal is to create a ``good'' coordinate
chart near $x_0$ on the leaf passing through $x_0$ generated by $Z_1,\ldots, Z_q$
(note that this leaf is independent of $\delta\in \q(0,1\w)^\nu$, but
we wish to have good estimates on this chart in terms of $Z_1,\ldots, Z_q$;
in particular, the theorem will be meaningful even if $Z_1,\ldots, Z_q$ span
the tangent space and the leaf is the entire ambient space).

For $\xi>0$, $x_0\in \R^n$, define
\begin{equation*}
\begin{split}
B_Z\q(x_0,\xi\w) = \bigg\{y \: \bigg|\: &\exists \gamma:\q[0,1\w]\rightarrow \R^n, \gamma\q(0\w) = x_0, \gamma\q(1\w)=y,\\
&\gamma'\q(t\w) = \sum_{j=1}^q a_j\q(t\w) Z_j\q(\gamma\q(t\w)\w), \q|a_j\q(t\w)\w|< \xi, \: \forall t
\bigg\};
\end{split}
\end{equation*}
$B_Z\q(x_0,\xi\w)$ is called a Carnot-Carath\'eodory ball.  Note that
$B_Z\q(x_0,\xi\w)$ is an open neighborhood of the point $x_0$ on the leaf
passing through $x_0$ generated by $Z_1,\ldots, Z_q$.
Let $n_0=\dim\Span{Z_1\q(x_0\w),\ldots, Z_q\q(x_0\w)}$ (the dimension
of this leaf).
The quantitative Frobenius theorem we will use is
\begin{thm}[\cite{StreetMultiParameterCCBalls}]\label{ThmFrob}
There exist $\xi,\eta\approx 1$ and a smooth map
$\Phi:B^{n_0}\q(\eta\w)\rightarrow B_Z\q(x_0,1\w)$ such that:
\begin{itemize}
\item $\Phi$ is one-to-one.
\item $B_Z\q(x_0,\xi\w)\subseteq \Phi\q(B^{n_0}\q(\eta\w)\w)$.
\item If we let $Y_j$ be the pull back of $Z_j$ via the map $\Phi$ to
$B^{n_0}\q(\eta\w)$, then $Y_1,\ldots, Y_q$ span the tangent space.
More precisely,
\begin{equation*}
1\lesssim 
\inf_{u\in B^{n_0}\q(\eta\w)} 
\sup_{1\leq j_1,\ldots, j_{n_0}\leq q} 
\q|\det\q[Y_{j_1}\q(u\w) : \cdots : Y_{j_{n_0}}\q(u\w)\w]\w|.
\end{equation*}
\item $Y_1,\ldots, Y_q$ are smooth.  That is,
\begin{equation*}
\CjN{Y_j}{m}{B^{n_0}\q(\eta\w)}\lesssim 1.
\end{equation*}
\end{itemize}
In the above, all implicit constants can be chosen to depend only
on $n$, upper bounds for a finite number of the $C^m$ norms
of the $Z_j$ and $\ct_{j,k}^l$, and $q$.
In particular, when $Z_j=\delta^{d_j} X_j$ as in our primary example,
the implicit constants are independent of $\delta$.
\end{thm}
Notice that the implicit constants in Theorem \ref{ThmFrob} do {\it not}
depend on lower bounds for quantities like
\begin{equation*}
\q|\det_{n_0\times n_0} \q[Z_1\q(x_0\w):\cdots: Z_q\q(x_0\w)\w]\w|,
\end{equation*}
where $\det_{n_0\times n_0} A$ denotes the vector whose coordinates
are the determinates of the $n_0\times n_0$ submatricies of $A$.
This is what sets Theorem \ref{ThmFrob} apart from classical versions
of the Frobenius theorem.

In particular, recall the sets $\sP$ and $\sN$ of vector fields 
with formal degrees
corresponding
to the pure powers and non-pure powers, respectively.
From the finite type condition we obtain the vector fields
$\sF = \q\{\q(X_1,d_1\w),\ldots, \q(X_q,d_q\w)\w\}$.
For each $j_0\in \N^\nu$, and $x_0\in \R^n$ ($x_0$ near $0$),
we consider the vector fields
$$Z_1=2^{-j_0\cdot d_1}X_1,\ldots, Z_q=2^{-j_0\cdot d_q}X_q,$$
which satisfy the hypotheses of Theorem \ref{ThmFrob}
uniformly in $j_0,x_0$.  We therefore obtain the map
$\Phi=\Phi_{x_0,j_0}$ as in Theorem \ref{ThmFrob}.
Let $Y_1,\ldots, Y_q$ be the pull backs of $Z_1,\ldots, Z_q$
so that $Y_1,\ldots, Y_q$ span the tangent space uniformly in $j_0,x_0$.

Further, we let $Y_\alpha$ be the pull back of $2^{-j_0\cdot \alpha} X_\alpha$,
where
\begin{equation*}
\gamma_t\q(x\w) = \exp\q(\sum_{0<\q|\alpha\w|\leq L} t^\alpha X_\alpha\w) x.
\end{equation*}
By \eqref{EqnFiniteTypeControl} and \eqref{EqnAlgControl} $Y_\alpha$
is a finite linear combination, with coefficients in $C^\infty$ of
the $Y_l$.  Thus $Y_\alpha\in C^\infty$ uniformly in $j_0$.
Furthermore, since each $Y_l$ is in the Lie algebra generated
by $Y_\alpha$ where $\alpha$ is a pure power (this comes from the corresponding
statement about the $X_l$), and since the $Y_l$ span the tangent
space uniformly in $j_0$, we have that $\q\{Y_\alpha : \alpha\text{ is a pure power}\w\}$ satisfies H\"ormander's condition, uniformly in $j_0$.

Let $j,k\in \N^\nu$ and set $j_0=j\wedge k\in \N^\nu$.  I.e.,
the $\mu$th coordinate of $j_0$ is the minimum of the $\mu$th coordinates
of $j$ and $k$.
Consider,
\begin{equation*}
\gamma_{2^{-k}t}\circ\gamma_{2^{-j}s}^{-1}\q(x\w) = 
\exp\q(\sum_{0<\q|\alpha\w|\leq L} t^{\alpha} 2^{-k\cdot \deg\q(\alpha\w)} X_\alpha\w)
\exp\q(-\sum_{0<\q|\alpha\w|\leq L} s^{\alpha} 2^{-j\cdot \deg\q(\alpha\w)} X_\alpha\w)x.
\end{equation*}
Hence, if we define,
\begin{equation*}
\theta_{t,s}\q(u\w) = \Phi_{x_0,j_0}^{-1}\circ \gamma_{2^{-k} t}\circ \gamma_{2^{-j}s}^{-1}\circ \Phi_{x_0,j_0}\q(u\w),
\end{equation*}
then,
\begin{equation}\label{EqnThetaAsExp}
\theta_{t,s}\q(u\w) = 
\exp\q(\sum_{0<\q|\alpha\w|\leq L} t^{\alpha} 2^{-\q(k-j_0\w)\cdot \deg\q(\alpha\w)} Y_\alpha\w)
\exp\q(-\sum_{0<\q|\alpha\w|\leq L} s^{\alpha} 2^{-\q(j-j_0\w)\cdot \deg\q(\alpha\w)} Y_\alpha\w)u.
\end{equation}
We have suppressed the dependence of $\theta$ on $j,k,x_0$.

Recall, $\alpha$ is a pure power if $\deg\q(\alpha\w)$ is non-zero in only
one component.  Thus, if $\alpha$ is a pure power,
either $\q(k-j_0\w)\cdot \deg\q(\alpha\w)$ or $\q(j-j_0\w)\cdot \deg\q(\alpha\w)$
will be $0$.
It follows that each vector field $Y_\alpha$ corresponding to a pure power
appears {\it unscaled} in at least one of the two exponentials
in \eqref{EqnThetaAsExp}.  Hence, the vector fields in the exponentials
of \eqref{EqnThetaAsExp} (namely, $\q\{2^{-\q(k-j_0\w)\cdot \deg\q(\alpha\w)}Y_\alpha, 2^{-\q(j-j_0\w)\cdot \deg\q(\alpha\w)}Y_\alpha\w\}$) satisfy H\"ormander's condition
uniformly in $j,k,x_0$.
This argument is the main reason we had to distinguish between
pure powers and non-pure powers.

The way $\theta$ comes up in our argument is as follows,
\begin{equation*}
T_j^{*} T_k f\q(x\w) = \psi\q(x\w) \int f\q(\gamma_{2^{-k}t}\circ \gamma_{2^{-j}s}^{-1}\q(x\w)\w) \kappa\q(s,t,x\w) \overline{\eta_j\q(s\w)} \eta_k\q(t\w)\:ds\: dt,
\end{equation*}
where $\kappa\in C^\infty$.  Thus, if we let $\Phi_{x_0,j_0}^{\#}$ denote the 
map $\Phi_{x_0,j_0}^{\#} f = f\circ \Phi_{x_0,j_0}$, then we see,
\begin{equation}\label{EqnTjTkPullback}
\Phi_{x_0,j_0}^{\#} T_j^{*} T_k \q(\Phi_{x_0,j_0}^{\#}\w)^{-1} g \q(u\w) = 
\psi\q(\Phi_{x_0,j_0}\q(u\w)\w) \int g\q(\theta_{t,s}\q(u\w)\w) \kapt\q(s,t,u\w) \overline{\eta_j\q(s\w)}\eta_k\q(t\w)\: ds\: dt.
\end{equation}
\eqref{EqnTjTkPullback} alone is not quite enough to directly
show \eqref{EqnToShowAO}, however it is close.
We outline the proof of the $L^2$ theorem in Section \ref{SectionL2}.

\section{The $L^2$ theorem}\label{SectionL2}
In this section, we outline the proof of the case $p=2$
of Theorem \ref{ThmMainThm} (for the operator $T$).
This describes work from
\cite{StreetMultiParameterSingRadonLt}.

As discussed in Section \ref{SectionScaling}, our goal is to show
\begin{equation}\label{EqnToShowAO2}
\LtOpN{T_j^{*} T_k}, \LtOpN{T_j T_k^{*}} \lesssim 2^{-\epsilon\q|j-k\w|},
\end{equation}
where $T_j$ and $T_k$ are defined in \eqref{EqnDefnTj}.
We focus on the estimate for $T_j^{*}T_k$, the other estimate being
similar.
In what follows $\epsilon>0$ is a constant that may change from
line to line, but is always independent of $j,k$.
To prove \eqref{EqnToShowAO2} it suffices to show
\begin{equation}\label{EqnToShowAO3}
\LtOpN{\q(T_k^{*}T_j T_j^{*} T_k\w)^{n+1}} \lesssim 2^{-\epsilon\q|j-k\w|},
\end{equation}
which proves \eqref{EqnToShowAO2} with $\epsilon$ replaced by $\frac{\epsilon}{2\q(n+1\w)}$.  Recall, $n$ is the dimension of the ambient space.
It is easy to see that
\begin{equation*}
\LpOpN{1}{T_k^{*} T_j T_j^{*} T_k}\lesssim 1,
\end{equation*}
and therefore to prove \eqref{EqnToShowAO3}, interpolation shows that it suffices to show
\begin{equation}\label{EqnToShowAO4}
\LpOpN{\infty}{\q(T_k^{*}T_j T_j^{*} T_k\w)^{n+1}}\lesssim 2^{-\epsilon \q|j-k\w|}.
\end{equation}

Fix a point $x_0$ and a bounded function $f$.  To show
\eqref{EqnToShowAO4}
we show that
\begin{equation}\label{EqnToShowAO5}
\q| \q(T_k^{*} T_j T_j^{*} T_k \w)^{n+1} f\q(x_0\w)\w| \lesssim 2^{-\epsilon\q|j-k\w|} \LpN{\infty}{f}.
\end{equation}
Recall
\begin{equation}\label{EqnRecallTjAO}
T_j f\q(x\w) =\psi\q(x\w) \int f\q(\gamma_t\q(x\w)\w) \dil{\eta_j}{2^j}\q(t\w) \: dt,
\end{equation}
where $\eta_j\in C_0^\infty\q(B^N\q(a\w)\w)$ (and satisfies certain
cancellation conditions); with a similar formula for $T_k$.  Here,
$a>0$ is small if the support of $K$ is small, and therefore, we may
take $a>0$ as small as we like (since the Theorem \ref{ThmMainThm}
is about kernels $K$ with small support).

Since $T_k^{*}$ and $T_j^{*}$ are essentially of the same form
as $T_k$ and $T_j$ (with $\gamma_t$ replaced by $\gamma_t^{-1}$), by taking $a>0$ small enough,
\eqref{EqnRecallTjAO} shows that 
$ \q(T_k^{*} T_j T_j^{*} T_k \w)^{n+1} f\q(x_0\w)$
depends only on the values of $f$ in $B_{Z}\q(x_0,\xi\w)$
where $j_0=j\wedge k$, $Z_1=2^{-j_0\cdot d_1} X_1,\ldots, Z_q=2^{-j_0\cdot d_q} X_q$, and $\xi$ is as in Theorem \ref{ThmFrob}.
Thus to prove \eqref{EqnToShowAO5} it suffices to show
\begin{equation*}
\q|\Phi_{x_0,j_0}^{\#} \q(T_k^{*} T_j T_j^{*} T_k \w)^{n+1} \q(\Phi_{x_0,j_0}^{\#}\w)^{-1} g\q(0\w) \w|\lesssim 2^{-\epsilon\q|j-k\w|} \LpN{\infty}{g},
\end{equation*}
where $g$ is supported on $B^{n_0}\q(\eta\w)$ (and $n_0$, $\eta>0$ are
as in Theorem \ref{ThmFrob}).

Putting all this together, to prove \eqref{EqnToShowAO2} it suffices to show
\begin{equation}\label{EqnToShowAO6}
\LpOpN{\infty}{\q(\Phi_{x_0,j_0}^{\#} T_k^{*} T_j T_j^{*} T_k \q(\Phi_{x_0,j_0}^{\#}\w)^{-1}\w)^{n+1}}\lesssim 2^{-\epsilon\q|j-k\w|},
\end{equation}
where $\epsilon>0$ and the implicit constant are independent
of $j$, $k$, and $x_0$.
Fortunately, \eqref{EqnToShowAO6} follows from
methods developed in the single parameter case
by Christ, Nagel, Stein, and Wainger
\cite{ChristNagelSteinWaingerSingularAndMaximalRadonTransforms}.

We close this section by briefly outlining
the proof of \eqref{EqnToShowAO6};
we refer the reader to
\cite{StreetMultiParameterSingRadonLt} for more precise details.
By a computation like the one leading up to \eqref{EqnTjTkPullback}
and an application of the Campbell-Hausdorff formula
it is easy to see that
\begin{equation*}
\Phi_{x_0,j_0}^{\#} T_k^{*} T_j T_j^{*} T_k \q(\Phi_{x_0,j_0}^{\#}\w)^{-1} g\q(u\w) = \psit\q(u\w) \int f\q(\Theta_t\q(u\w)\w) \kapt\q(t,u\w)\etat\q(t\w) \: dt,
\end{equation*}
where $\etat\in C_0^\infty$ and is supported near $0\in \R^{4N}$, $\kapt\in C^\infty$, and\footnote{\eqref{EqnThetaAsym}
means that $\Theta_t \q(u\w)  = \exp\q(\sum_{0<\q|\alpha\w|\leq L} t^{\alpha} W_\alpha\w)u
+O\q(\q|t\w|^{L+1}\w)$ as $t\rightarrow 0$.}
\begin{equation}\label{EqnThetaAsym}
\Theta_t \q(u\w) \sim \exp\q(\sum_{0<\q|\alpha\w|} t^{\alpha} W_\alpha\w)u,
\end{equation}
and $\q\{W_\alpha\w\}$ satisfy H\"ormander's condition, uniformly
in any relevant parameters.
Using results from \cite{ChristNagelSteinWaingerSingularAndMaximalRadonTransforms}
it is then easy to show that
\begin{equation}\label{EqnSmoothAfterIter}
\q[\Phi_{x_0,j_0}^{\#} T_k^{*} T_j T_j^{*} T_k \q(\Phi_{x_0,j_0}^{\#}\w)^{-1} \w]^{n} g\q(u\w) =  \psit\q(u\w) \int g\q(v\w) h\q(u,v\w)\: dv,
\end{equation}
where $h\q(u,v\w)$ has a certain level of ``smoothness'' in $v$.
Indeed, one can show that (uniformly in $u$),
$h\q(u,\cdot\w)\in L^1_\delta$ for some $\delta>0$, where
$L^1_\delta$ is the space of those $f\in L^1$ such that,
\begin{equation*}
\int \q|f\q(y-z\w) -f\q(y\w)\w| \: dy\lesssim \q|z\w|^{\delta}.
\end{equation*}

There is also ``cancellation'' in the operator
$\Phi_{x_0,j_0}^{\#} T_k^{*} T_j T_j^{*} T_k \q(\Phi_{x_0,j_0}^{\#}\w)^{-1}$
induced by the cancellation condition in $\eta_j,\eta_k$.
One can use this cancellation together with the smoothness
from \eqref{EqnSmoothAfterIter} to show that
\begin{equation*}
\LpOpN{\infty}{\q[\Phi_{x_0,j_0}^{\#} T_k^{*} T_j T_j^{*} T_k \q(\Phi_{x_0,j_0}^{\#}\w)^{-1} \w]^{n} \Phi_{x_0,j_0}^{\#} T_k^{*} T_j T_j^{*} T_k \q(\Phi_{x_0,j_0}^{\#}\w)^{-1}}\lesssim 2^{-\epsilon\q|j-k\w|},
\end{equation*}
which completes the proof of \eqref{EqnToShowAO6} and therefore
of \eqref{EqnToShowAO2}.

\begin{rmk}
A significant issue dealt with is related to the fact that
$T_j$ and $T_k^{*}$ do not necessarily commute.  The proof
of \eqref{EqnToShowAO6} exploits this possibility, allowing
us to take vectors in \eqref{EqnThetaAsExp} from {\it both}
exponentials to obtain a collection that satisfies H\"ormander's
condition.
This allows one to exploit, for instance, the commutator
of two vector fields, one from each exponential.
This is in contrast to the case when $\nu=1$, where one needs
only look at one of the exponentials.  This accounts for one
of the differences between our cited work
and \cite{ChristNagelSteinWaingerSingularAndMaximalRadonTransforms}.
Of course, there are case where $T_j$ and $T_k^{*}$ (approximately) commute (for instance, certain cases when all the $X_\alpha$ commute), and our theorem
applies.  In these cases, one does not
need to exploit the commutator of two vector fields, one from
each of the exponentials in \eqref{EqnThetaAsExp}.
\end{rmk}

\section{The Littlewood-Paley theory}\label{SectionLittlewood}
In order to extend the results from $L^2$ to $L^p$ ($1<p<\infty$)
we use a Littlewood-Paley theory adapted to the operator $T$.
Recall the vector fields with formal degrees $\q(X_1, d_1\w),\ldots, \q(X_q,d_q\w)$
from Section \ref{SectionScaling} (these vector fields are nothing
more than an enumeration of the set $\sF$ from the finite-type condition).

For each $\mu$, $1\leq \mu\leq \nu$, consider those vector fields with formal
degrees $\q(X_j,d_j\w)$ such that $d_j$ is non-zero in only the $\mu$th
component.  
Think of these vector fields as vector fields with {\it one}-parameter
formal degrees, where the one parameter is the $\mu$th component
of $d_j$.
Enumerate these vector fields 
and one-parameter formal degrees
$\q(X_1^{\mu}, d_1^{\mu}\w),\ldots,\q(X_{q_\mu}^{\mu},d_{q_\mu}^\mu\w)$
where $d_j^\mu\in \q(0,\infty\w)$.

For $\q(t_1,\ldots, t_{q_\mu}\w)\in \R^{q_\mu}$ and $j\in \Z$, define 
$2^{j} \q(t_1,\ldots, t_{q_\mu}\w)= \q(2^{jd_1^\mu} t_1,\ldots, 2^{jd_{q_\mu}^\mu}t_{q_\mu}\w)$.
For a function $f\q(t\w)$ (with $t\in \R^{q_\mu}$), define
\begin{equation*}
\dil{f}{2^j} = 2^{j\q(d_1^\mu+\cdots+d_{q_\mu}^{\mu}\w)} f\q(2^j t\w).
\end{equation*}
Decompose
\begin{equation*}
\delta_0 = \sum_{j=0}^\infty \dil{\eta_j}{2^j}\q(t\w),
\end{equation*}
where $\q\{\eta_j\w\}\subset C_0^\infty\q(B^{q_\mu}\q(a\w)\w)$ is a bounded
set and $\int \eta_j=0$ if $j\ne 0$ (here, $a>0$ is a small number).

Fix $\psi_0\in C_0^\infty\q(\R^n\w)$ which is $1$ on a neighborhood of
$0$ and supported on a small neighborhood of $0$.  Define,
\begin{equation*}
D_j^\mu f\q(x\w) = \psi_0\q(x\w) \int_{\R^{q_\mu}} f\q(e^{t\cdot X^\mu} x\w) \psi_0\q(e^{t\cdot X^\mu}x\w)\dil{\eta_j}{2^j}\q(t\w) \: dt,
\end{equation*}
where $X^\mu=\q(X^\mu_1,\ldots, X_{q_\mu}^{\mu}\w)$.  Note that $\sum_{j=0}^\infty D_j^\mu = \psi_0^2$.

For $j=\q(j_1,\ldots, j_\nu\w)\in \N^\nu$, define,
\begin{equation*}
D_j = D_{j_1}^1\cdots D_{j_\nu}^\nu.
\end{equation*}
Note that $\sum_{j\in \N^\nu} D_j = \psi_0^{2\nu}$.

\begin{thm}[Littlewood-Paley square function]\label{ThmSquare}
For $f$ supported sufficiently close to $0\in \R^n$ and $1<p<\infty$,
\begin{equation*}
\LpN{p}{\q(\sum_{j\in \N^\nu} \q|D_j f\w|^2\w)^{\frac{1}{2}}}\approx \LpN{p}{f}.
\end{equation*}
\end{thm}

For $M\in \N$, define
\begin{equation*}
U_M f\q(x\w) = \sum_{j\in \N^\nu} \sum_{\substack{k\in \N^\nu\\ \q|k-j\w|\leq M}} D_j D_k.
\end{equation*}

\begin{thm}[Calder\'on reproducing formula]\label{ThmReproduce}
Fix $\psi_1\prec \psi_0$, $\psi_1\in C_0^\infty$.
For every $p$ ($1<p<\infty$) there exists an $M$, and an operator $V_M:L^p\rightarrow L^p$ such that $\psi_1 U_M V_M = \psi_1 = V_M U_M\psi_1 $.
\end{thm}
\begin{proof}[Proof sketch]
Note that
$\psi_0^{4\nu} = U_M +R_M$,
where $$R_M=\sum_{j\in \N^\nu} \sum_{\substack{k\in \N^\nu \\ \q|k-j\w|>M}} D_j D_k.$$
For each fixed $p$, $1<p<\infty$, one has,
\begin{equation}\label{EqnRMSmall}
\lim_{M\rightarrow \infty} \LpOpN{p}{R_M}=0.
\end{equation}
Taking $M=M\q(p\w)$ such that $\LpOpN{p}{R_M}$ is sufficiently small,
$V_M$ can be constructed using a Neumann series.
For $p=2$, \eqref{EqnRMSmall} follows from the ideas
in Section \ref{SectionL2}.  For other $p$, an interpolation argument is used.
\end{proof}

The use of Theorem \ref{ThmReproduce} is that instead of proving the
$L^p$ boundedness of $T$, it suffices to prove the $L^p$ boundedness
of $T U_M$ for some $M=M\q(p\w)$ (here we have used $T\psi_1=T$ if
the support of $T$ is chosen to be on a sufficiently small neighborhood of $0$--where we have taken $\psi_1\equiv 1$ on a neighborhood of $0$).

We close this section by outlining the proof of the $\lesssim$
part of Theorem \ref{ThmSquare}.  We defer further discussion of all
other results in this section to
\cite{SteinStreetMultiParameterSingRadonLp}.

By a well-known reduction, to prove the $\lesssim$ part of
Theorem \ref{ThmSquare} it suffices to show that for
any $\nu$ sequences $\epsilon_j^\mu\in \q\{-1,1\w\}$, $j\in \N$, $1\leq \mu\leq \nu$, we have
\begin{equation*}
\sum_{j_1,\ldots, j_\nu=0}^\infty \epsilon_{j_1}^1 D_{j_1}^1 \cdots \epsilon_{j_\nu}^\nu D_{j_\nu}^\nu,
\end{equation*}
is bounded on $L^p$ (with bound independent of the choice of the sequences $\epsilon^\mu$).
As a consequence, it suffices to show for each $\mu$,
\begin{equation*}
S_\mu := \sum_{j=0}^\infty \epsilon_{j}^\mu D_j^\mu
\end{equation*}
is bounded on $L^p$ (with bound independent of the sequence).

If the vector fields $X_1^\mu,\ldots, X_{q_\mu}^{\mu}$ span
the tangent space at each point, then it is easy to see
that $S_\mu$ is a Calder\'on-Zygmund singular integral operator
associated to single parameter Carnot-Carath\'eodory balls
generated by $\q(X_1^\mu, d_1^\mu\w),\ldots,\q(X_{q_\mu}^\mu, d_{q_\mu}^\mu\w)$
(which, in turn, endow the ambient space with the structure of a space
of homogeneous type--see \cite{NagelSteinWaingerBallsAndMetricsDefinedByVectorFields}).  The $L^p$ boundedness of $S_\mu$ then
follows from classical results (see, e.g., \cite{SteinHarmonicAnalysis}).

It does not follow from our assumptions that
$X_1^\mu,\ldots, X_{q_\mu}^{\mu}$ span the tangent space.
However, it {\it does} follow from our assumptions
that $\q[X_j^\mu, X_k^\mu\w]$ can be written as a $C^\infty$
linear combination of $X_1^\mu,\ldots, X_{q_\mu}^\mu$ and
therefore the classical theorem of Frobenius applies
to foliate the ambient space into leaves; $X_1^\mu, \ldots, X_{q_\mu}^\mu$
spanning the tangent space to each leaf.
One would then like to apply the above argument (seeing $S_\mu$
as a Calder\'on-Zygmund singular integral operator)
to each leaf (uniformly in the leaf), and use this to
prove the $L^p$ boundedness of $S_\mu$ on the entire space.
To do this, one needs appropriately uniform control of the coordinate
charts defining the leaves.  Theorem \ref{ThmFrob}
gives just such control, and this argument then yields the $L^p$
boundedness of $S_\mu$.

Along a similar line,
let $\sigma\in C_0^\infty\q(\R^{q_\mu}\w)$ be a
non-negative function which is $1$ on a neighborhood
 of $0$, and supported on a small neighborhood of $0$.
Define,
\begin{equation*}
\begin{split}
A_j^\mu f\q(x\w) &= 
\psi_0\q(x\w) \int f\q(e^{t\cdot X^\mu} x\w) \psi_0\q(e^{t\cdot X^\mu} x\w) \dil{\sigma}{2^j}\q(t\w)\: dt\\
&=\psi_0\q(x\w) \int f\q(e^{\q(2^{-j}t\w)\cdot X^\mu} x\w) \psi_0\q(e^{\q(2^{-j}t\w)\cdot X^\mu} x\w) \sigma\q(t\w)\: dt,
\end{split}
\end{equation*}
and define
\begin{equation*}
\sM_\mu f\q(x\w) = \sup_{j\in \N} A_j^\mu \q|f\w| \q(x\w).
\end{equation*}
By a similar argument using the Frobenius theorem
to reduce the question to the classical theory of spaces of
homogeneous type,
we have $\LpN{p}{\sM_\mu f}\lesssim \LpN{p}{f}$ for $1<p\leq \infty$.

\section{Auxiliary operators}
To prove Theorem \ref{ThmMainThm}, we need to use several
auxiliary operators.
Recall, we have decomposed $T=\sum_{j\in \N^{\nu}} T_j$,
we have defined a Littlewood-Paley square
function in terms of operators $D_j$ in Section
\ref{SectionLittlewood}, and we have introduced some
``averaging'' operators $A_j^\mu$ at the end of
Section \ref{SectionLittlewood}.
In what follows, we will have many operators with a subscript
with values in $\N^\nu$ (e.g., $T_j$, $j\in \N^\nu$).  For such operators
we define $T_j$ for $j\in \Z^\nu\setminus \N^\nu$ to be $0$.

For $k_1,k_2\in \Z^\nu$, we define the operator $\sT_{k_1,k_2}$
on sequences of measurable functions $\q\{f_j\w\}_{j\in \N^\nu}$ by
\begin{equation*}
\sT_{k_1,k_2} \q\{f_j\w\}_{j\in \N^\nu} = \q\{ D_j T_{j+k_1} D_{j+k_2} f_j\w\}_{j\in \N^\nu}.
\end{equation*}
The use of $\sT_{k_1,k_2}$ comes from the following proposition:
\begin{prop}\label{PropToShowVectT}
Fix $p$, $1<p<\infty$.  If there exists $\epsilon>0$ such that
\begin{equation*}
\LplqOpN{p}{2}{\sT_{k_1,k_2}}\lesssim 2^{-\epsilon\q(\q|k_1\w|+\q|k_2\w|\w)},
\end{equation*}
then $T:L^p\rightarrow L^p$.  Here, $L^p\q(\ell^q\q(\N^\nu\w)\w)$ denotes the
$L^p$ space of functions on $\R^n$ taking values in the Banach
space $\ell^q\q(\N^\nu\w)$.
\end{prop}
\begin{proof}
Take $M=M\q(p\w)$ as in Theorem \ref{ThmReproduce}, so that it
suffices to prove $TU_M:L^p\rightarrow L^p$.
Consider, using Theorem \ref{ThmSquare} and the triangle inequality,
we have
\begin{equation*}
\begin{split}
\LpN{p}{ T U_M f} &\approx \LpN{p}{\q(\sum_{j\in \N^\nu} \q|D_j T U_M f\w|^2\w)^{\frac{1}{2}}}\\
&= \LpN{p}{\q(  \sum_{j\in \N^\nu} \q|\sum_{\substack{j_1,j_2\in \N^\nu\\ \q|l\w|\leq M}} D_j T_{j_1} D_{j_2} D_{j_2+l} f\w|^2         \w)^{\frac{1}{2}}} \\
&= \LpN{p}{\q(  \sum_{j\in \N^\nu} \q|\sum_{\substack{k_1,k_2\in \Z^\nu\\ \q|l\w|\leq M}} D_j T_{j+k_1} D_{j+k_2} D_{j+k_2+l} f\w|^2         \w)^{\frac{1}{2}}} \\
&\leq \sum_{\substack{k_1,k_2\in \Z^\nu\\ \q|l\w|\leq M} }\LpN{p}{\q(  \sum_{j\in \N^\nu} \q|D_j T_{j+k_1} D_{j+k_2} D_{j+k_2+l} f\w|^2         \w)^{\frac{1}{2}}} \\
&=  \sum_{\substack{k_1,k_2\in \Z^\nu\\ \q|l\w|\leq M} } \LplqN{p}{2}{\sT_{k_1,k_2} \q\{D_{j+k_2+l} f\w\}_{j\in \N^\nu}}\\
&\lesssim \sum_{\substack{k_1,k_2\in \Z^\nu\\ \q|l\w|\leq M} } 2^{-\epsilon \q(\q|k_1\w| +\q|k_2\w|\w)} \LplqN{p}{2}{\q\{D_{j+k_2+l}f\w\}_{j\in \N^\nu}}\\
&\lesssim \LpN{p}{\q(\sum_{j\in \N^\nu} \q|D_j f\w|^2\w)^{\frac{1}{2}}}\\
&\lesssim \LpN{p}{f},
\end{split}
\end{equation*}
which completes the proof.
\end{proof}

Our proof of the $L^p$ boundedness of $T$ will proceed by applying
Proposition \ref{PropToShowVectT}.
A similar idea is used for $\sM$, however $\sM$ does not have the inherent
cancellation that $T$ does.  Because of this, we introduce
this cancellation in an {\it ad hoc} way.
To do this we proceed by induction on $\nu$ (the base case
$\nu=0$ will be trivial--this will be explained more in what follows).  We will construct an operator $B_j$ ($j\in \N^\nu$)
which has cancellation similar to $T_j$ and which (under our inductive
hypothesis) satisfies $\LpN{p}{\sup_j \q|B_j f\w|}\lesssim \LpN{p}{f}$ ($1<p<\infty$) if and only if $\LpN{p}{\sM f}\lesssim \LpN{p}{f}$.

Let $\Ninf =\N\cup\q\{\infty\w\}$.
For a subset $E\subseteq \q\{1,\ldots, \nu\w\}$ and $j=\q(j_1,\ldots, j_\nu\w)\in \N^\nu$,
define $j_E\in \Ninf^\nu$ to be equal to $j_\mu$ for those components $\mu\in E$
and equal to $\infty$ in the rest of the components.
For $t \in \R^N$, we define $2^{-j_E} t$ in the usual way, with the convention
that $2^{-\infty}=0$; thus $2^{-j_E}t$ is zero in every coordinate
corresponding to $t_\mu$ where $\mu\in E^c$.
We may think of these dilations as $\q|E\w|$-parameter dilations
acting on the lower dimensional space consisting of those
coordinates $t_\mu$ with $\mu\in E$.  Notice that
$j_\emptyset = \q(\infty,\cdots,\infty\w)$ and $j_{\q\{1,\ldots, \nu\w\}}=j$.

Let $\sigma_0\in C_0^\infty\q(\R\w)$ be a non-negative function supported on a small neighborhood
of $0$, with $\sigma_0\geq 1$ on a neighborhood of $0$.
For $t=\q(t_1,\ldots, t_m\w)\in \R^m$ (for any $m$), we define
$\sigma\q(t\w) = \sigma_0\q(t_1\w)\cdots \sigma_0\q(t_m\w)$.
We use this $\sigma$ to define the operators $A_j^\mu$ as at
the end of Section \ref{SectionLittlewood}.
Let $\psi_1,\psi_2\in C_0^\infty\q(\R^n\w)$, with $\psi_1,\psi_2\geq 0$, $\psi_2\prec \psi_1\prec \psi_0$,
and $\psi_1,\psi_2\geq 1$ on a neighborhood of $0$.
We define, for $j\in \Ninf^\nu$,
\begin{equation*}
M_j f\q(x\w) =\psi_2\q(x\w) \int f\q(\gamma_{2^{-j}t}\q(x\w)\w) \psi_2\q(\gamma_{2^{-j}t}\q(x\w)\w) \sigma\q(t\w) \: dt.
\end{equation*}
Notice,
\begin{equation}\label{EqnBaseCase}
M_{\q(\infty,\cdots,\infty\w)} f\q(x\w) = M_{j_\emptyset} f\q(x\w) = \psi_2^2\q(x\w) \q[\int \sigma\q(t\w)\: dt\w]f\q(x\w).
\end{equation}
It is immediate to see
\begin{equation*}
\sM f\q(x\w) \lesssim \sup_{j\in \N^\nu} M_j \q|f\w| \q(x\w).
\end{equation*}
Thus, to prove $\sM$ is bounded on $L^p$ ($1<p\leq \infty$), it suffices to show
\begin{equation}\label{EqnToShowsM}
\LpN{p}{\sup_{j\in \N^\nu} \q|M_j f\w|} \lesssim \LpN{p}{f},
\end{equation}
for $1<p\leq \infty$.

Recall the operators $A_j^\mu$ from Section \ref{SectionLittlewood} (replace
$\psi_0$ with $\psi_1$ in the definition of $A_j^\mu$).
For $j=\q(j_1,\ldots, j_\nu\w)\in \Ninf^\nu$, define
\begin{equation*}
A_j = A_{j_1}^1 \cdots A_{j_\nu}^\nu.
\end{equation*}
Note that
\begin{equation*}
A_{\q(\infty,\cdots,\infty\w)} = \q[\int \sigma\q(t\w) \: dt\w]^\nu \psi_1^{2\nu}.
\end{equation*}
Thus, to prove \eqref{EqnToShowsM} it suffices to show
\begin{equation}\label{EqnToShowsM2}
\LpN{p}{\sup_{j\in \N^\nu} \q|A_{j_\emptyset}M_{j_{\q\{1,\ldots, \nu\w\}}} f\w|} \lesssim \LpN{p}{f}.
\end{equation}

It is easy to see that $M_{j_E}$ is of the same form as $M_j$ but
with $\nu$ replaced by $\q|E\w|$.  Thus, we may assume (for induction)
that $\LpN{p}{\sup_{j\in \N^\nu} \q|M_{j_E} f\w|}\lesssim \LpN{p}{f}$
for $1<p\leq \infty$ and $E\subsetneq \q\{1,\ldots, \nu\w\}$.
In light of \eqref{EqnBaseCase}, the base case $E=\emptyset$
is trivial.

Fix $p$ ($1<p\leq 2$).\footnote{The maximal result for $p>2$ follows from
the result for $p=2$ and
interpolation with the trivial result $p=\infty$.}
As discussed at the end of Section \ref{SectionLittlewood},
\begin{equation*}
\LpN{p}{\sup_{j\in \Ninf^{\nu}} \q|A_j f\w|}\lesssim \LpN{p}{f}.
\end{equation*}
It follows that, for $E\subsetneq \q\{1,\ldots, \nu\w\}$,
\begin{equation*}
\LpN{p}{\sup_{j\in \N^\nu} \q|A_{j_{E^c}} M_{j_E} f\w|} \lesssim \LpN{p}{f}.
\end{equation*}
Define the operator,
\begin{equation*}
B_j = \sum_{E\subseteq \q\{1,\ldots, \nu\w\}} \q(-1\w)^{\q|E\w|} A_{j_{E^c}} M_{j_E}.
\end{equation*}

By the above discussion, showing \eqref{EqnToShowsM} is equivalent to showing
\begin{equation*}
\LpN{p}{\sup_{j\in \N^\nu} \q|B_j f\w|} \lesssim \LpN{p}{f}.
\end{equation*}
In fact, we show the stronger result:
\begin{equation}\label{EqnToShowBj}
\LpN{p}{\q(\sum_{j\in \N^\nu} \q|B_j f\w|^2\w)^{\frac{1}{2}}}\lesssim \LpN{p}{f}.
\end{equation}
For $k\in \Z^\nu$, define the vector valued operator
\begin{equation*}
\sB_k \q\{f_j\w\}_{j\in \N^\nu} = \q\{B_j D_{j+k} f_j\w\}_{j\in \N^\nu}.
\end{equation*}
From the above discussion, and a proof along the lines of Proposition
\ref{PropToShowVectT}, we have the following.
\begin{prop}\label{PropToShowVectM}
Fix $p$, $1<p<\infty$.  If there exists $\epsilon>0$ such that
\begin{equation*}
\LplqOpN{p}{2}{\sB_k}\lesssim 2^{-\epsilon\q|k\w|},
\end{equation*}
then $\sM$ is bounded on $L^p$.
\end{prop}

\section{Completion of the proof}
In this section, we outline the remaining points
in the proof of Theorem \ref{ThmMainThm}.
Since the $L^2$ adjoint of $T$ is essentially of the same
form as $T$ (with $\gamma_t\q(x\w)$ replaced by $\gamma_t^{-1}\q(x\w)$)
and since $\sM$ is trivially bounded on $L^\infty$, it suffices
to verify Theorem \ref{ThmMainThm} for $1<p\leq 2$.
In light of Propositions \ref{PropToShowVectT} and \ref{PropToShowVectM}
we wish to show that for every $p$, $1<p\leq 2$, there exists
$\epsilon=\epsilon\q(p\w)>0$ such that
\begin{equation}\label{EqnToShowVectBound}
\LplqOpN{p}{2}{\sT_{k_1,k_2}}\lesssim 2^{-\epsilon\q(\q|k_1\w|+\q|k_2\w|\w)}, \quad \LplqOpN{p}{2}{\sB_{k}}\lesssim 2^{-\epsilon \q|k\w|}.
\end{equation}

\begin{step}
Using the methods of Section \ref{SectionL2}, we show
\begin{equation*}
\LpOpN{2}{D_jT_{j+k_1}D_{j+k_2}}\lesssim 2^{-\epsilon_2 \q(\q|k_1\w|+\q|k_2\w|\w)}, \quad \LpOpN{2}{B_j D_{j+k}} \lesssim 2^{-\epsilon_2\q|k\w|},
\end{equation*}
for some $\epsilon_2>0$.
As a consequence, we obtain
\begin{equation}\label{EqnToShowVectBoundL2}
\LplqOpN{2}{2}{\sT_{k_1,k_2}}\lesssim 2^{-\epsilon_2\q(\q|k_1\w|+\q|k_2\w|\w)}, \quad \LplqOpN{2}{2}{\sB_{k}}\lesssim 2^{-\epsilon_2 \q|k\w|}.
\end{equation}
It is easy to see that
\begin{equation}\label{EqnVectBoundL1}
\LplqOpN{1}{1}{\sT_{k_1,k_2}}\lesssim 1, \quad \LplqOpN{1}{1}{\sB_{k}}\lesssim 1.
\end{equation}
Interpolating \eqref{EqnToShowVectBoundL2} and \eqref{EqnVectBoundL1}
we obtain for $1<p\leq 2$,
\begin{equation}\label{EqnToVectBoundLp}
\LplqOpN{p}{p}{\sT_{k_1,k_2}}\lesssim 2^{-\epsilon_p\q(\q|k_1\w|+\q|k_2\w|\w)}, \quad \LplqOpN{p}{p}{\sB_{k}}\lesssim 2^{-\epsilon_p \q|k\w|},
\end{equation}
where $\epsilon_p>0$.
\end{step}

\begin{step}
\eqref{EqnToShowVectBoundL2} shows that $\sM$ is bounded on $L^2$.
A bootstrapping argument like the one from \cite{NagelSteinWaingerDifferentiationInLacunaryDirections} then can
be used (in conjunction with \eqref{EqnToVectBoundLp}) to show that $\sM$ is bounded on $L^p$ ($1<p\leq 2$), thereby
completing the proof for $\sM$.
\end{step}

\begin{step}
The $L^p$ boundedness of $\sM$ can be used to show that
\begin{equation*}
\LplqOpN{p}{\infty}{\sT_{k_1,k_2}} \lesssim 1.
\end{equation*}
Interpolating this with \eqref{EqnToVectBoundLp} yields \eqref{EqnToShowVectBound}
and completes the proof.
\end{step}

\section{More general results}\label{SectionGeneral}
The discussion in this note has been restricted to $\gamma$
of the form
\begin{equation*}
\gamma_t\q(x\w) = \exp\q(\sum_{0<\q|\alpha\w|\leq L} t^\alpha X_\alpha\w)x,
\end{equation*}
an exponential of a {\it finite} sum of vector fields.  However,
our cited work is
not limited to such $\gamma$.

If $\gamma$ were of the form
\begin{equation}\label{EqnGenExp}
\gamma_t\q(x\w) = \exp\q( A\q(t\w)\w) x,
\end{equation}
where $A\q(t\w)$ is a vector field depending smoothly on $t$ with $A\q(0\w)\equiv 0$,
then the discussion in this note would be easy to generalize.
Unfortunately, not every $\gamma$ is of this form.
Fortunately, there is a simple alternative.
Given $\gamma$ define a vector field depending on $t$, $W\q(t\w)$, by
\begin{equation}\label{EqnDefnW}
W\q(t\w)=W\q(t,x\w) = \frac{d}{d\epsilon}\bigg|_{\epsilon=1} \gamma_{\epsilon t}\circ \gamma_t^{-1}\q(x\w)\in T_x\R^n.
\end{equation}
Note that $W\q(0\w)\equiv 0$.  Moreover, the map $\gamma\mapsto W$
is a bijection (in an appropriate sense).
Everywhere in this note where one might wish to use $A$,
it suffices to use $W$ instead.

We now informally state a special case of the main result of
\cite{StreetMultiParameterSingRadonLt,
SteinStreetMultiParameterSingRadonLp},
but refer the reader to those papers for a more rigorous
statement in full generality.
The setting is the same as Theorem \ref{ThmMainThm}, and
so we are given a decomposition
$\R^N=\R^{N_1}\times \cdots \times \R^{N_\nu}$.
We are interested in the $L^p$ boundedness of operators of the form
\begin{equation*}
Tf\q(x\w) = \psi\q(x\w) \int f\q(\gamma_t\q(x\w)\w) K\q(t\w)\: dt,
\end{equation*}
where $K$ is a product kernel relative to this decomposition
(with small support).  $\gamma$ is a $C^\infty$ function
satisfying $\gamma_0\q(x\w)\equiv x$, but we do not
assume that $\gamma$ is given by \eqref{EqnGenExp}.

For a set of vector fields $\sV$, let $\sD\q(\sV\w)$ denote the
involutive distribution generated by $\sV$; that is, let
$\sD\q(\sV\w)$ be the smallest $C^\infty$ module containing
$\sV$ and closed under Lie brackets.

Given $\gamma$ define $W$ by \eqref{EqnDefnW}.
Decompose $W$ as a Taylor series in $t$,
\begin{equation*}
W\q(t\w) \sim \sum_{0<\q|\alpha\w|} t^{\alpha} X_\alpha,
\end{equation*}
where each $X_\alpha$ is a $C^\infty$ vector field.
We define pure powers and non-pure powers just as in
Section \ref{SectionResults}.
Our assumptions are (informally stated) as follows:
\begin{enumerate}[(i)]
\item For every $\delta\in\q(0,1\w]^\nu$,
\begin{equation*}
\sD_\delta = \sD\q(\q\{\delta_1^{\q|\alpha_1\w|} \cdots \delta_\nu^{\q|\alpha_\nu\w|} X_\alpha:\alpha \text{ is a pure power}\w\}\w)
\end{equation*}
is finitely generated as a $C^\infty$ module ``uniformly'' in $\delta$.

\item $W\q(\delta t\w)=W\q(\delta_1 t_1,\ldots, \delta_\nu t_\nu\w) \in \sD_\delta$ ``uniformly'' in $\delta\in \q(0,1\w]^\nu$.
\end{enumerate}

\begin{rmk}
Note that, if it were not for the ``uniform'' part of the above
assumptions, they would be independent of $\delta$.  Thus it is the
uniform part which is the heart of the above assumptions.
We refer the reader to \cite{StreetMultiParameterSingRadonLt}
for a more detailed discussion of these assumptions.
\end{rmk}

\begin{thm}
Under the above assumptions (made precise in \cite{StreetMultiParameterSingRadonLt})
the operator $T:L^p\rightarrow L^p$ ($1<p<\infty$).
\end{thm}

\begin{rmk}
There is also a corresponding maximal result.  See
\cite{SteinStreetMultiParameterSingRadonLp}.
\end{rmk}

\begin{rmk}
If $\gamma_t\q(x\w) = \exp\q(A\q(t\w)\w)x$, where $A\q(0\w)\equiv 0$,
then the above assumptions can be (equivalently) stated
with $W$ replaced by $A$.  The point of using $W$ is that that
every $\gamma$ has a corresponding $W$, while not every $\gamma$
has a corresponding $A$.
\end{rmk}

\begin{rmk}
In the single parameter case ($\nu=1$), the assumptions of
\cite{ChristNagelSteinWaingerSingularAndMaximalRadonTransforms}
can be equivalently stated as: the $\q\{X_\alpha\w\}$ satisfy
H\"ormander's condition.  This is a special case of the above assumptions.
If the $X_\alpha$ do not satisfy H\"ormander's condition, the above assumptions
still hold (in the case $\nu=1$) if:
\begin{enumerate}[(i)]
\item The involutive distribution
generated by the $\q\{X_\alpha\w\}$ is finitely generated as
a $C^\infty$-module.
\item $W$ is tangent to the leaves of the foliation corresponding
to this distribution in an appropriate ``scale invariant'' way.
\end{enumerate}
When the $\q\{X_\alpha\w\}$ satisfy H\"ormander's condition, this
involutive distribution is the entire space of smooth sections
of the tangent bundle (and is therefore clearly finitely generated)
and there is only one leaf (the entire space).  Thus, $W$ is
trivially tangent to this leaf (the scale invariance turns out
to follow as well).
\end{rmk}

\begin{rmk}
We have not separated the above conditions into analogs
of the finite type and algebraic conditions
discussed in Section \ref{SectionResults}, however
this can be done.  See
\cite{StreetMultiParameterSingRadonAnal}.
\end{rmk}

There is another way in which our cited work is more general
than what is discussed here:  we deal with a larger class
of singular kernels $K$ than product kernels.
We defer a general discussion to \cite{StreetMultiParameterSingRadonLt},
and instead discuss a special case that illustrates the ideas.
The setting is the three dimensional Heisenberg group $\Ho$.
As a manifold $\Ho$ is $\R^3$ and we write $\q(x,y,t\w)$ for
the coordinates of $\Ho$.
For $j=\q(j_1,j_2\w)\in \Z$ and $\q(x,y,t\w)\in \Ho$, write
$2^j \q(x,y,t\w)=\q(2^{j_1} x, 2^{j_2} y, 2^{j_1+j_2}t\w)$; and
for a function $f\q(x,y,t\w)$ write $\dil{f}{2^j}\q(x,y,t\w)= 2^{2j_1+2j_2} f\q(2^{j_1}x,2^{j_2}y, 2^{j_1+j_2}t\w)$.

Let $\q\{\eta_j\w\}_{j\in \Z^2}\subseteq C_0^\infty\q(B^3\q(1\w)\w)$ be a bounded set
such that
\begin{equation*}
\int \eta_j\q(x,y,t\w) \: dx\: dt=0=\int \eta_j\q(x,y,t\w)\: dy\: dt, \quad \forall j\in \Z^2.
\end{equation*}
Let $K$ be a distribution defined by
\begin{equation*}
K\q(x,y,t\w)=\sum_{j\in \Z^2} \dil{\eta_j}{2^j}\q(x,y,t\w).
\end{equation*}
\begin{prop}\label{PropToy1}
The operator defined by $f\mapsto f*K$ is bounded on $L^p$ ($1<p<\infty$)
where the convolution is taken in the sense of $\Ho$.
\end{prop}

The idea of Proposition \ref{PropToy1} is the following.  We may write
$f*K$ as
\begin{equation*}
f*K\q(\xi\w) = \int f\q(e^{xX +yY+tT} \xi\w) K\q(x,y,t\w)\: dt,
\end{equation*}
where $X,Y,T$ are the usual left invariant vector fields on $\Ho$
satisfying $\q[X,Y\w]=T$ and $T$ is in the center of the Lie algebra.
Because of our dilations we assign the formal degrees
$\q(X,\q(1,0\w)\w)$, $\q(Y,\q(0,1\w)\w)$, and $\q(T, \q(1,1\w)\w)$;
so that $T$ is a vector field corresponding to a non-pure power.
Note that
$\q(T,\q(1,1\w)\w)= \q(\q[X,Y\w],\q(1,0\w)+\q(0,1\w)\w)$ and so the algebraic
condition is satisfied.  The finite type condition trivially holds.

\begin{rmk}
If the convolution in Proposition \ref{PropToy1} is replaced with the convolution using the usual
group structure of $\R^3$, then
there are kernels $K$ (of the above type) such that the above operator is not bounded on $L^p$
for any $p$.
Moreover, the algebraic condition is not satisfied.  Indeed,
if the usual Euclidean group structure is used,
then $X$, $Y$, and $T$ would be replaced with
$\frac{\partial}{\partial x}$, $\frac{\partial}{\partial y}$, 
and $\frac{\partial}{\partial t}$, respectively.
Since $\q[\frac{\partial}{\partial x},\frac{\partial}{\partial y}\w]=0$,
the algebraic condition fails.
\end{rmk}

\bibliographystyle{amsalpha}

\bibliography{radon}





\end{document}